\documentclass[12pt]{amsart}
\usepackage[T1]{fontenc}
\usepackage[utf8]{inputenc}
\usepackage{lmodern}
\usepackage{amsmath,amssymb,mathtools}
\usepackage[hmargin=1.05in,vmargin=0.9in]{geometry}
\usepackage{microtype}
\usepackage{url}
\usepackage[hidelinks]{hyperref}
\hypersetup{
  pdftitle={Pseudodifferential structure, non-rigidity, and sharpness for the Steklov problem on differential forms},
  pdfsubject={Differential forms and boundary spectral geometry},
  pdfkeywords={Steklov eigenvalues, differential forms, Dirichlet-to-Neumann operator, conformal invariance, pseudodifferential operators},
  pdflang={en}
}
\numberwithin{equation}{section}

\newtheorem{theorem}{Theorem}[section]
\newtheorem{proposition}[theorem]{Proposition}
\newtheorem{lemma}[theorem]{Lemma}
\newtheorem{corollary}[theorem]{Corollary}
\theoremstyle{remark}
\newtheorem{remark}[theorem]{Remark}

\newcommand{\bt}{\mathbf{t}}
\newcommand{\bn}{\mathbf{n}}
\newcommand{\R}{\mathbb R}
\newcommand{\cC}{\mathrm c\mathcal C}
\newcommand{\HD}{\mathcal H_D}
\newcommand{\Id}{\operatorname{Id}}
\newcommand{\ii}{\mathrm i}
\newcommand{\Psicl}[1]{\Psi_{\mathrm{cl}}^{#1}}
\newcommand{\ip}[3]{\langle\!\langle #1,#2\rangle\!\rangle_{L^2(#3)}}
\newcommand{\doi}[1]{\href{https://doi.org/#1}{doi:\nolinkurl{#1}}}
\DeclareMathOperator{\im}{im}
\DeclareMathOperator{\Vol}{Vol}
\DeclareMathOperator{\Int}{int}

\title{Pseudodifferential structure, rigidity, and sharpness
for the Steklov problem on differential forms}
\author{Daguang Chen}
\email{dgchen@tsinghua.edu.cn}
\author{Guotao Huang}
\email{hgt26@mails.tsinghua.edu.cn}
\author{Tingyu Liao}
\email{liaoty25@mails.tsinghua.edu.cn}
\address{Department of Mathematical Sciences, Tsinghua University, Beijing 100084, P. R. China.}
\thanks{The authors were supported by  NSFC-FWO W2521103.}
\date{\today}

\begin{document}

\begin{abstract}
We resolve four questions concerning Karpukhin's Dirichlet-to-Neumann
operator on differential forms. First, for $0\leq p\leq n-2$, we prove
that the operator is classical pseudodifferential of order one and compute
its principal symbol, which is elliptic on the coclosed symbol space.
This answers Open Question~7.19 in the survey of Colbois, Girouard,
Gordon, and Sher~\cite{CGGS2024}. Second, the symbol calculation yields a Weyl law with
remainder $O(\tau^{n-2})$, resolving the smooth case of their Open
Question~7.23 in \cite{CGGS2024}. Third, in the critical dimension $n=2p+2$, we show that
conformal deformations of the Euclidean ball fixing the boundary metric
preserve the full operator. Nonisometric metrics therefore retain all
equalities in Karpukhin's initial equality range, disproving the
isometry-rigidity assertion of Conjecture~2.12 in \cite{Karpukhin2019}(the rigidity assertion of Conjecture~7.27 in \cite{CGGS2024}). Fourth, we establish the
all-index sharpness asserted in Conjecture~2.13 in \cite{Karpukhin2019}(the sharpness assertion of Conjecture~7.27 in \cite{CGGS2024}) in the class of smooth
compact connected oriented Riemannian manifolds with boundary. For
$p\geq1$, a shrinking-core construction inspired by the collapsing models
of Ann\'e and Takahashi produces manifolds with fixed spherical boundary
and explicit two-sided bounds for an initial eigenvalue cluster.

\end{abstract}

\maketitle
\markboth{DAGUANG CHEN, GUOTAO HUANG, AND TINGYU LIAO }{ FOUR QUESTIONS
FOR KARPUKHIN'S STEKLOV OPERATOR }

\section{Introduction}\label{sec:introduction}
Let $(M^n,g)$ be a smooth compact connected oriented Riemannian manifold,
$n\geq2$, with smooth nonempty boundary $\Sigma$, and let
$g_\Sigma=g|_{\Sigma}$. The classical Steklov problem is
\begin{equation}
  \Delta_g u=0 \quad\text{in }M,\qquad
  \partial_\nu u=\sigma u \quad\text{on }\Sigma,
\end{equation}
where $\nu$ is the outward unit normal derivative. Its eigenvalues
form a discrete sequence
\[
  0=\sigma_0(M,g)<\sigma_1(M,g)\leq\sigma_2(M,g)\leq\cdots
  \longrightarrow\infty.
\]
For background and recent developments, see
\cite{BucFNT21,ChenH2026,CGGS2024,GaoWeiZhou2026,GuLiWan2025,GuLiWan2026,KwongWei2023,LevitinMangoubiPolterovich2023,LiWangWu2025,XiaXiong2024,Xiong2018,Xiong2021,Xiong2022,YangYu2017JMP}.

For differential forms, several different Dirichlet-to-Neumann operators
extend the scalar problem, depending on the boundary conditions imposed
on the extension. These include the map of Belishev and Sharafutdinov
\cite{BelishevSharafutdinov2008}, motivated by inverse problems and
Maxwell's equations, the full boundary map studied by Joshi and
Lionheart \cite{JoshiLionheart2005}, and the elliptic self-adjoint
operator of Raulot and Savo \cite{RaulotSavo2012,RaulotSavo2014}.
Karpukhin \cite{Karpukhin2019} modified the Belishev--Sharafutdinov
construction by using coclosed harmonic extensions with free normal
trace. The resulting operator annihilates exact boundary forms and has
a nonnegative self-adjoint realization with compact resolvent on
coclosed forms.  Using a weighted Reilly formula, Xiong \cite{Xiong2026Reilly}
obtained a sharp lower bound for the first positive eigenvalue of
$\Lambda_g$ and comparisons with the boundary Hodge spectrum under
curvature and convexity assumptions. Throughout this paper, we use Karpukhin's Steklov operator $\Lambda_g$, and its Steklov eigenvalues are those of
its restriction to the coclosed boundary space
$\cC^p(\Sigma)$.

In this paper, we resolve four problems concerning the Karpukhin's Steklov operator $\Lambda_g$: \textbf{1).} its pseudodifferential
structure, \textbf{2).} the remainder in its Weyl law, \textbf{3).} equality rigidity, \textbf{4).} 
all-index sharpness of the critical-degree Hersch--Payne--Schiffer
inequality. These are Open Questions~7.19 and~7.23 and the two assertions
of Conjecture~7.27 in the survey of Colbois, Girouard, Gordon, and Sher
\cite{CGGS2024}. The pseudodifferential question was already noted by Karpukhin \cite[Remark~2.4]{Karpukhin2019}. Girouard, Karpukhin, Levitin, and Polterovich established the Weyl law in \cite[Theorem~5.8]{GKLP2022} and discussed improving its remainder in \cite[Remark~5.9]{GKLP2022}. The rigidity and sharpness assertions were originally formulated by Karpukhin \cite[Conjectures~2.12 and~2.13]{Karpukhin2019}.

To formulate these problems and our results precisely, we first recall
the definition of $\Lambda_g$. Following notations in \cite{Karpukhin2019}, denote
by $\nu$ the outward unit normal and by $i:\Sigma\hookrightarrow M$ the
inclusion. For $\varphi\in\Omega^p(\Sigma)$, choose an extension $\omega$
satisfying 
\begin{equation}\label{eq:extension}
  \Delta_g\omega=0,\qquad \delta_g\omega=0,\qquad i^*\omega=\varphi,
\end{equation}
and define
\begin{equation}\label{eq:Lambda}
  \Lambda_g\varphi=i^*\iota_\nu d\omega.
\end{equation}
The value of \eqref{eq:Lambda} is independent of the choice of extension.
For $0\leq p\leq n-2$, write
\[
  \cC^p(\Sigma)=\ker\delta_\Sigma,\qquad
  0\leq\sigma_1^{(p)}(M,g)\leq\sigma_2^{(p)}(M,g)\leq\cdots
\]
for the smooth coclosed space and the spectrum of $\Lambda_g$ on its
$L^2$ closure. Eigenvalues are counted with multiplicity, including zeros.
We use the same convention for the boundary Hodge eigenvalues
$\lambda_k^{\prime(p)}(\Sigma,g_{\Sigma})$ on coclosed $p$-forms. Set
\begin{equation}\label{eq:topological-indices}
  I_p(M)=\dim\im\bigl(i^*:H^p(M;\R)\longrightarrow H^p(\Sigma;\R)\bigr),
  \qquad b_q(M)=\dim H^q(M;\R).
\end{equation}

The classical inequalities of Hersch, Payne, and Schiffer
\cite{HerschPayneSchiffer1975} were extended to higher-dimensional
complementary form degrees by Yang and Yu \cite{YangYu2017} and
Karpukhin \cite{Karpukhin2019}. Yang and Yu use the operator of
Raulot--Savo  \cite{RaulotSavo2012,RaulotSavo2014} ; Karpukhin uses $\Lambda_g$. In dimension $n=2p+2$,
Karpukhin states that \cite[Theorem~2.11]{ Karpukhin2019}
\begin{equation}\label{eq:HPS}
  \bigl(\sigma_{m+I_p(M)}^{(p)}(M,g)\bigr)^2
  \leq
  \lambda_{I_p(M)+b_{p+1}(M)+2m-1}^{\prime(p)}(\Sigma,g_{\Sigma}),
  \qquad m\geq1.
\end{equation}
where $\lambda^{\prime(p)}_k$ is the $k$-th eigenvalue of the Hodge-Laplacian on the space $c\mathcal{C}^p(\Sigma).$
Throughout this paper, we denote
\begin{equation}\label{eq:Dp}
  D_p=\binom{2p+2}{p+1},\qquad N=\left\lceil\frac{2m-1}{D_p}\right\rceil.
\end{equation}
The Euclidean unit ball realizes equality in \eqref{eq:HPS} for every
$1\leq m\leq D_p/2$. The special role of this dimension comes from the 
conformal invariance of the energy of the $(p+1)$-form $d\omega$.

\subsection{The four questions.}
The four questions addressed in this paper concern Karpukhin's Steklov operator
$\Lambda_g$, denoted by $\mathcal D_p^K$ in \cite{CGGS2024}.

\begin{enumerate}
  \renewcommand{\labelenumi}{(\roman{enumi})}

  \item \emph{Pseudodifferential structure.}
  Open Question~7.19 in \cite{CGGS2024} asks whether $\Lambda_g$ is a
  pseudodifferential operator. We prove that it is classical of order
  one and compute its principal symbol. Although the full operator is
  not elliptic in positive degree, its principal symbol is elliptic
  on the coclosed symbol space.

  \item \emph{The Weyl remainder.}
  Open Question~7.23 in \cite{CGGS2024} asks whether the remainder in
  the Weyl law for the coclosed Steklov counting function can be
  improved to $O(\tau^{n-2})$. We establish this estimate for smooth
  metrics and smooth boundaries as a consequence of the symbol
  calculation.

  \item \emph{Equality rigidity.}
  In dimension $n=2p+2$, Conjecture~2.12 of
  \cite{Karpukhin2019} asserts that equality in the critical-degree
  Hersch--Payne--Schiffer inequality throughout the initial range
  $1\leq m\leq D_p/2$ forces the manifold to be a Euclidean ball.
  We disprove this isometry-rigidity assertion by conformal
  deformations of the ball that preserve the boundary metric and
  the full operator $\Lambda_g$.

  \item \emph{Sharpness at every index.}
  Conjecture~2.13 of \cite{Karpukhin2019} asks whether the same
  inequality is asymptotically sharp for every index $m$ and every
  degree $p$. We prove this assertion in the class of smooth compact
  connected oriented Riemannian manifolds with boundary. For
  $p\geq1$, the proof uses a shrinking core with attached Euclidean
  annuli; the case $p=0$ follows from the established planar result
  of Girouard and Polterovich \cite{GirouardPolterovich2010}.
\end{enumerate}

The last two problems are, respectively, the rigidity and sharpness
assertions collected in \cite[ Conjecture~7.27]{CGGS2024}. The sharpness result allows
several boundary components and does not require the interior metric
to be globally Euclidean.

\subsection{Main results}

We denote by $\Psicl{m}(\Sigma;E,F)$ (or simply $\Psicl{m}(\Sigma;E)$ when $E=F$) the space of classical  pseudodifferential operators of order at most $m$ between smooth vector bundles $E$ and $F$ over the closed manifold $\Sigma$. An operator $P\in\Psicl{m}$ possesses a complete symbol with an asymptotic expansion in local coordinates $\sigma(P)(x,\xi)\sim\sum_{j=0}^\infty p_{m-j}(x,\xi)$, where each component $p_{m-j}$ is positively homogeneous of degree $m-j$ in $\xi$ for $|\xi|\geq1$, and its leading term $\sigma_m(P)=p_m$ is the principal symbol of order $m$. We write $\Psi^{-\infty}$ for the class of smoothing operators with smooth Schwartz kernels.

Let $\Pi_p$ be the orthogonal Hodge projection onto coclosed boundary $p$-forms. For $(x,\xi)\in T^*\Sigma\setminus0$, write $|\xi|=|\xi|_{g_{\Sigma}}$ and
\begin{equation}\label{eq:pi}
  \pi_p(x,\xi)=\Id-|\xi|^{-2}\xi\wedge\iota_{\xi^\sharp}.
\end{equation}
This is the orthogonal projection onto $\ker\iota_{\xi^\sharp}$.

\begin{theorem}\label{thm:pdo}
Let $(M^n,g)$ be a smooth compact connected oriented Riemannian manifold of dimension $n\geq2$ with smooth nonempty boundary $\Sigma$, let $g_{\Sigma}=g|_{\Sigma}$ be the induced boundary metric, and let $0\leq p\leq n-2$. Then
\begin{equation}\label{eq:pdo-main}
  \Lambda_g\in\Psicl{1}(\Sigma;\Lambda^pT^*\Sigma),\qquad
  \sigma_1(\Lambda_g)(x,\xi)=|\xi|\pi_p(x,\xi).
\end{equation}
Moreover, $\Pi_p\in\Psicl{0}$ has principal symbol $\pi_p$, and
\begin{equation}\label{eq:compression-main}
  \Lambda_g=\Pi_p\Lambda_g\Pi_p.
\end{equation}
There is a projected parametrix $\mathcal Q_p\in\Psicl{-1}$ such that
\begin{equation}\label{eq:parametrix-main}
  \mathcal Q_p\Lambda_g=\Pi_p+S_1,\qquad
  \Lambda_g\mathcal Q_p=\Pi_p+S_2,\qquad
  S_1,S_2\in\Psi^{-\infty}.
\end{equation}
The coclosed realization is nonnegative and self-adjoint, with domain
$H^1\Omega^p(\Sigma)\cap\ker\delta_\Sigma$ and compact resolvent.
\end{theorem}

For $p\geq1$, the full symbol in \eqref{eq:pdo-main} has a kernel; the
full operator is not elliptic. Its restriction to the symbol range of
$\Pi_p$ is $|\xi|\Id$. Thus Theorem~\ref{thm:pdo} answers Open~Question~7.19 in  \cite{CGGS2024} affirmatively and supplies the precise projected ellipticity
suggested in \cite[Remark~2.4]{Karpukhin2019}.

The proof expresses $\Lambda$ as a Schur complement of the full
Dirichlet boundary map for the Hodge Laplacian. A finite-dimensional
correction treats the kernel of the boundary operator used to impose
$\delta\omega=0$. In particular, no vanishing cohomology assumption
on $M$ is required.

As observed in \cite{GKLP2022}, the explicit computation of the principal symbol yields the Weyl asymptotics for the coclosed Steklov spectrum. Moreover, the pseudodifferential structure established here gives the sharp \(O(\tau^{n-2})\) remainder, thereby resolving the smooth-boundary case of Open Question~7.23 in \cite{CGGS2024}.
Here and throughout, $\omega_n$ denotes the volume of the Euclidean unit $n$-ball.

\begin{corollary}\label{cor:Weyl}
Let $0\leq p\leq n-2$.
For the eigenvalues of the coclosed realization,
\begin{equation}\label{eq:Weyl}
  \#\{j:\sigma_j^{(p)}(M,g)<\tau\}
  =\binom{n-2}{p}\frac{\omega_{n-1}\Vol_{g_{\Sigma}}(\Sigma)}{(2\pi)^{n-1}}\tau^{n-1}
   +O(\tau^{n-2})\qquad(\tau\longrightarrow\infty).
\end{equation}
In particular, the smooth case of Open Question~7.23 in \cite{CGGS2024} has an
affirmative answer.
\end{corollary}

Our next result shows that the conformal invariance proved in
\cite[Proposition~2.10]{Karpukhin2019} is incompatible with the proposed
isometry rigidity.

\begin{theorem}\label{thm:nonrigidity}
Let $p\geq0$, $n=2p+2$, and let $g_E$ be the Euclidean metric on the
closed unit ball $\mathbb{B}^n$. If $u\in C^\infty(\mathbb{B}^n)$ and $u|_{\partial \mathbb{B}^n}=0$,
then
\begin{equation}\label{eq:ball-invariance}
  \Lambda_{e^{2u}g_E}=\Lambda_{g_E}.
\end{equation}
The boundary Hodge spectrum is unchanged, and equality in
\eqref{eq:HPS} holds for every $1\leq m\leq D_p/2$.
If $0\ne u\in C_c^\infty(\Int \mathbb{B}^n)$ and $u\geq0$, then
$(\mathbb{B}^n,e^{2u}g_E)$ is not isometric to any Euclidean ball, although it
agrees with the Euclidean metric on an entire boundary collar.
For a fixed such $u$, the metrics $e^{2tu}g_E$, $t>0$, are pairwise
nonisometric.
\end{theorem}

Hence Conjecture~2.12 of Karpukhin, also assertion (1) of
Conjecture~7.27 in \cite{CGGS2024}, are false in the Riemannian category. The counterexamples
preserve the complete boundary operator and therefore all its spectral
data.

The sharpness assertion has a positive answer. The quantitative statement
below is independent of \eqref{eq:HPS}.

\begin{theorem}\label{thm:cluster}
Fix $p\geq1$, $N\geq1$ and put $n=2p+2$. There exists a family of smooth compact connected
oriented Riemannian manifolds $(M_\epsilon,g_\epsilon)$,
$0<\epsilon<1$, such that
\begin{equation}\label{eq:cluster-boundary}
  (\partial M_\epsilon,g_\epsilon|_{T\partial M_\epsilon})
  \cong\coprod_{j=1}^N(\mathbb{S}^{2p+1},g_{\mathbb{S}}),\qquad
  I_p(M_\epsilon)=b_{p+1}(M_\epsilon)=0,
\end{equation}
where $\cong$ denotes an isometry and $g_{\mathbb{S}}$ is the unit round metric.
With $T=\log(1/\epsilon)$,
\begin{equation}\label{eq:cluster-bounds}
  (p+1)\tanh((p+1)T)\leq\sigma_j^{(p)}(M_\epsilon,g_\epsilon)
  \leq(p+1)\coth((p+1)T),\qquad 1\leq j\leq ND_p.
\end{equation}
In particular, the first $ND_p$ eigenvalues converge to $p+1$ while the
entire boundary metric remains fixed.
\end{theorem}

\begin{corollary}\label{cor:sharpness}
For every $p\geq1$ and $m\geq1$. 
The family in Theorem~\ref{thm:cluster} satisfies
\begin{equation}\label{eq:deficit-main}
\begin{split}
  \lambda_{2m-1}^{\prime(p)}(\partial M_\epsilon)&=(p+1)^2,\\
  0\leq (p+1)^2-\bigl(\sigma_m^{(p)}(M_\epsilon,g_\epsilon)\bigr)^2
  &\leq\frac{4(p+1)^2\epsilon^{2(p+1)}}{(1+\epsilon^{2(p+1)})^2}.
\end{split}
\end{equation}
Thus the Hersch-Payne-Schiffer inequality \eqref{eq:HPS} is asymptotically sharp at every index in every
positive degree. Together with the planar result of Girouard and
Polterovich in \cite{GirouardPolterovich2010}, this establishes sharpness for all $p\geq0$ and $m\geq1$.
\end{corollary}

The comparison eigenvalue in \eqref{eq:deficit-main} is fixed and positive.
The family has $N$ boundary components, and its connecting core carries
a smooth Riemannian metric. The assertion is made in the Riemannian class
specified above; sharpness for Euclidean domains with connected boundary
requires a separate argument.

\subsection{Method and relation to previous work}

The operator proof starts with the full Dirichlet-to-Neumann map for
harmonic forms studied by Joshi and Lionheart \cite{JoshiLionheart2005}.
Solving the coclosedness constraint gives an exact boundary-operator
formula. Its symbol identifies both the degeneracy on exact forms and
the ellipticity on the coclosed symbol space. This also gives a short
proof of the smooth Weyl remainder requested in \cite[Open
Question~7.23]{CGGS2024}; see Corollary~\ref{cor:Weyl}. Girouard, Karpukhin, Levitin,
and Polterovich \cite{GKLP2022} had obtained a comparison with the
boundary Hodge spectrum and a Weyl law with a little-$o$ remainder.

The geometric construction follows the scaled-core and conical-collar
models of Ann\'e and Takahashi
\cite{AnneTakahashi2012,AnneTakahashi2015,AnneTakahashi2024}.
We attach $N$ identical Euclidean annuli to a shrinking core and keep
their outer spheres as the Steklov boundary. The estimates concern the
minimized boundary energy $\int_M|d\omega|^2\,dV$. In the critical
dimension this energy becomes a product-cylinder energy. The coclosed
Hodge projection removes the exact tangential contribution, and the two
scalar endpoint problems give the factors $(p+1)\tanh((p+1)T)$ and $(p+1)\coth((p+1)T)$.
Thus the spectral bounds follow without a general gluing convergence
theorem. The specific geometric models are described in
\cite[\S1.1]{AnneTakahashi2015} and
\cite[\S5.3]{AnneTakahashi2024}.

For $p=0$, all-index sharpness is the planar theorem of Girouard and
Polterovich \cite{GirouardPolterovich2010}, combined with fixed-domain
smoothing. The new geometric estimates concern $p\geq1$. The non-rigidity
result uses Karpukhin's existing conformal invariance; its content is the
resulting obstruction to his proposed equality characterization.

The paper is organized as follows. Section~\ref{sec:preliminaries} establishes the Hodge-theoretic framework and variational conventions used throughout. In Section~\ref{sec:pdo}, we prove Theorem~\ref{thm:pdo} and derive the corresponding sharp Weyl remainder estimate. Section~\ref{sec:conformal} is devoted to the proof of Theorem~\ref{thm:nonrigidity}. Finally, in Sections~\ref{sec:construction} and~\ref{sec:energy}, we construct the degenerating family, establish the required energy estimates, and prove Theorem~\ref{thm:cluster} and Corollary~\ref{cor:sharpness}.

\section{Preliminaries}\label{sec:preliminaries}

We use Karpukhin's notation in \cite{Karpukhin2019} and the nonnegative Hodge Laplacian
$\Delta=d\delta+\delta d$. All forms are smooth up to the boundary unless
a Sobolev space is specified. The geometric arguments may be carried out
over $\R$; we complexify the bundles and use Hermitian inner products
when using pseudodifferential symbols. Following Karpukhin,
$\langle\cdot,\cdot\rangle$ denotes the pointwise inner product, and
\[
  \ip{\alpha}{\beta}{M}
  =\int_M\langle\alpha,\beta\rangle\,dV_M
\]
denotes the integrated inner product.
For a smooth subspace $V$, $H^sV$ denotes its completion in the Sobolev
$H^s$ norm, and $L^2V=H^0V$.

\subsection{Traces and Green's formula}

For a form $\omega$ on $M$, the tangential and normal components are
$\bt\omega$ and $\bn\omega$. A boundary form is identified with its tangential
representative in $\Lambda^*T^*M|_\Sigma$, and
\[
  \bt\omega= i^*\omega,\qquad
  \bn\omega=\nu^\flat\wedge \iota_{\nu}\omega,\qquad
  \omega|_\Sigma=\bt\omega+\bn\omega.
\]
Thus the full boundary value is determined by the pair
$(i^*\omega,i_\nu\omega)$.

For $\alpha\in\Omega^q(M)$ and $\beta\in\Omega^{q+1}(M)$, Green's formula is
\begin{equation}\label{eq:Green}
  \ip{d\alpha}{\beta}{M}
  =\ip{\alpha}{\delta\beta}{M}
   +\ip{i^*\alpha}{i^*\iota_\nu\beta}{\Sigma}.
\end{equation}
Applying this identity twice gives, for $p$-forms $u,v$,
\begin{equation}\label{eq:Green-energy}
\begin{split}
  \ip{du}{dv}{M}+\ip{\delta u}{\delta v}{M}
  &=\ip{\Delta u}{v}{M}+\ip{i^*\iota_\nu du}{i^*v}{\Sigma}\\
  &\quad-\ip{i^*\delta u}{i^*\iota_\nu v}{\Sigma}.
\end{split}
\end{equation}
We distinguish the harmonic fields
\[
  \mathcal H^p(M)=\{u\in \Omega^p(M):du=0,\ \delta u=0\}
\]
from the solutions of $\Delta u=0$, which need not be closed or coclosed
on a manifold with boundary. The relative harmonic fields are
\[
  \HD^p(M)=\{u\in\mathcal H^p(M):i^*u=0\},\qquad \mathcal{H}_{N}^p(M) =\{ u\in \mathcal{H}^p(M) :\iota_\nu u=0 \}.
\]
The finite-dimensional spaces $\mathcal{H}_N^p(M)$ and $\mathcal{H}_D^p(M)$ are isomorphic to $H^p(M,\mathbb{R})$ and $H^p(M,\partial M,\mathbb{R})$, respectively. The
standard boundary Hodge theory is developed in \cite{Schwarz1995}.

We will use the following consequence of \eqref{eq:Green}, proved in
\cite[Proposition~3.9]{Karpukhin2019}.
\begin{lemma}\label{lem:coclosed}
If $\Delta u=0$ and $i^*\delta u=0$, then $\delta u=0$.
\end{lemma}

\subsection{The boundary energy}

The solutions of \eqref{eq:extension} form an affine space $\mathcal{L}(\varphi)$
with difference space $\HD^p(M)$. Orthogonality to $\HD^p(M)$ selects
the canonical extension $\lambda_g(\varphi)$
\cite[Proposition~3.11]{Karpukhin2019}. In particular,
$d\lambda_g(\varphi)$ is independent of the normalization.
For smooth boundary data define
\begin{equation}\label{eq:boundary-energy}
  q_g(\varphi,\psi)
  =\ip{d\lambda_g(\varphi)}{d\lambda_g(\psi)}{M}
  =\ip{\Lambda_g\varphi}{\psi}{\Sigma},\qquad
  q_g(\varphi)=q_g(\varphi,\varphi).
\end{equation}
The second identity follows from \eqref{eq:Green}. In particular,
$\Lambda_g$ is formally self-adjoint and nonnegative.
If $\eta$ is any smooth extension of $\varphi$ and
$\zeta=\eta-\lambda_g(\varphi)$, then $i^*\zeta=0$ and
$\delta d\lambda_g(\varphi)=0$, so
\[
  \ip{d\zeta}{d\lambda_g(\varphi)}{M}=0.
\]
Consequently,
\begin{equation}\label{eq:energy-minimum}
\begin{gathered}
  \|d\eta\|_{L^2(M)}^2=q_g(\varphi)+\|d\zeta\|_{L^2(M)}^2,\\
  q_g(\varphi)
  =\min_{\substack{\eta\in\Omega^p(M)\\i^*\eta=\varphi}}
       \int_M|d\eta|_g^2\,dV_g.
\end{gathered}
\end{equation}
There is no interior coclosedness condition on the trial extension
$\eta$ \cite[Proposition~5.1]{Karpukhin2019}.

We use Karpukhin's max--min characterization
\cite[Theorem~5.2]{Karpukhin2019}:
\begin{equation}\label{eq:variational}
  \sigma_k^{(p)}(M,g)
  =\max_{\substack{F\subset\cC^p(\Sigma)\\\dim F=k-1}}
    \ \min_{\substack{0\ne\varphi\in\cC^p(\Sigma)\\\varphi\perp F}}
       \frac{q_g(\varphi)}{\|\varphi\|_{L^2(\Sigma)}^2}.
\end{equation}
It is enough to take smooth finite-dimensional spaces. The energy form
extends to the closed form on $H^{1/2}\Omega^p(\Sigma)\cap\ker\delta_\Sigma$.
The kernel of the coclosed realization has dimension $I_p(M)$.
Orthogonality throughout is in the boundary $L^2$ inner product.

\subsection{The boundary Hodge projection}

On the closed manifold $\Sigma$, let $H_p$ be the orthogonal projection
onto harmonic $p$-forms and let $G_p$ be the Hodge Green operator.
The Hodge decomposition and Green-operator identities
\cite[Theorem~6.8, Definition~6.9, and Proposition~6.10]{Warner1983} give
\begin{equation}\label{eq:Green-Hodge}
\begin{gathered}
  \Delta_\Sigma G_p=G_p\Delta_\Sigma=\Id-H_p,\qquad
  d_\Sigma G_p=G_{p+1}d_\Sigma,\qquad
  \delta_\Sigma G_p=G_{p-1}\delta_\Sigma.
\end{gathered}
\end{equation}
The changes in degree in the commutation identities are understood.
The Hodge decomposition gives
\begin{equation}\label{eq:projection}
  \Pi_p=H_p+\delta_\Sigma d_\Sigma G_p
       =\Id-d_\Sigma\delta_\Sigma G_p.
\end{equation}
It is the $L^2$-orthogonal projection onto $\ker\delta_\Sigma$. In particular, for any $\alpha \in \Omega^p(\Sigma)$,
\begin{equation}\label{eq:projection-properties}
  d_\Sigma\Pi_p\alpha=d_\Sigma\alpha,\qquad
  \Pi_pd_\Sigma\alpha=0,\qquad
  \|\Pi_p\alpha\|_{L^2(\Sigma)}\leq\|\alpha\|_{L^2(\Sigma)}.
\end{equation}
All these identities hold componentwise if $\Sigma$ is disconnected.
At degree zero, terms involving negative-degree forms are zero and
$\Pi_0=\Id$.

\begin{lemma}\label{lem:compression}
For $0\leq p\leq n-2$,
\[
  \Lambda_g\Pi_p=\Pi_p\Lambda_g=\Lambda_g.
\]
\end{lemma}
\begin{proof}
If $p\geq1$ and $\varphi=d_\Sigma\gamma$, extend $\gamma$ smoothly to
$\widetilde\gamma$ on $M$. The form $d\widetilde\gamma$ extends
$\varphi$ and has zero exterior-derivative energy. Hence
\eqref{eq:energy-minimum} gives $d\lambda_g(\varphi)=0$ and
$\Lambda_g\varphi=0$. Since $\Id-\Pi_p$ projects onto exact forms,
this proves $\Lambda_g\Pi_p=\Lambda_g$. Formal self-adjointness implies
that the image of $\Lambda_g$ is orthogonal to exact forms, proving the
other identity. For $p=0$, the conclusion follows from $\Pi_0=\Id$.
\end{proof}

\subsection{Classical boundary operators}

We use $\Psicl{m}(E,F)$ for classical pseudodifferential operators of
order at most $m$ between smooth vector bundles on the closed boundary.
Composition adds orders and multiplies principal symbols. Elliptic
operators have parametrices, and an invertible classical elliptic operator
has a classical inverse of the opposite order; see \cite{Hormander2007}.

For full Dirichlet data
$(\varphi,\beta)\in\Omega^p(\Sigma)\oplus\Omega^{p-1}(\Sigma)$,
let $H(\varphi,\beta)$ be the unique solution of
\begin{equation}\label{eq:full-Dirichlet}
  \Delta u=0,\qquad i^*u=\varphi,\qquad i^*\iota_\nu u=\beta.
\end{equation}
This is the full Dirichlet problem for a Laplace-type operator:
both tangential and normal components are prescribed; see
\cite[Theorem~3.4.10]{Schwarz1995}. In inward normal coordinates
$(x',s)$, write
\begin{equation}\label{eq:normal-collar}
  g=ds^2+h_s,\qquad \nu=-\partial_s,\qquad \nu^\flat=-ds.
\end{equation}
Write the harmonic form in these coordinates as
\begin{equation}\label{eq:collar-form}
  u=a(s)-ds\wedge b(s),\qquad a(0)=\varphi,\qquad b(0)=\beta.
\end{equation}
Here $a(s)$ and $b(s)$ are tangential forms identified by the collar
coordinates. The componentwise outward-normal boundary map is
\begin{equation}\label{eq:full-DN}
  \mathcal N(\varphi,\beta)
  =\bigl(-\partial_sa|_{s=0},-\partial_sb|_{s=0}\bigr).
\end{equation}
The full Dirichlet-to-Neumann construction for the Hodge Laplacian gives
\begin{equation}\label{eq:full-DN-symbol}
  \mathcal N\in\Psicl{1},\qquad
  \sigma_1(\mathcal N)(x,\xi)=|\xi|_{g_{\Sigma}}\Id;
\end{equation}
see \cite[Theorem~1.1(i) and Proposition~3.1]{JoshiLionheart2005}.
The principal symbol uses the outward normal; in the frozen half-space
model the decaying solution is $e^{-s|\xi|}$, whose outward derivative
is $|\xi|$ at $s=0$.

\begin{remark}\label{rem:top-degree}
We restrict the infinite-dimensional spectral statements to
$0\leq p\leq n-2$. For completeness, under
\eqref{eq:extension}--\eqref{eq:Lambda}, the top boundary degree has the
explicit finite-rank formula
\begin{equation}\label{eq:top-degree}
  \Lambda_g\varphi
  =\frac{\displaystyle\int_\Sigma\varphi}{\Vol_n(M,g)}\,dA_{g_{\Sigma}},
  \qquad \varphi\in\Omega^{n-1}(\Sigma).
\end{equation}
Indeed, $d\omega=f\,dV_g$ and $\delta d\omega=0$ imply that $f$ is
constant on the connected manifold. Stokes' theorem determines $f$, and
normal contraction gives \eqref{eq:top-degree}. Thus this operator is
smoothing and has rank one. Since $dA_{g_{\Sigma}}$ is coclosed, restriction to
$\cC^{n-1}(\Sigma)$ does not make it zero. This degree does not enter the
Hersch--Payne--Schiffer comparison considered here.
\end{remark}

\section{Pseudodifferential structure}\label{sec:pdo}

We first prove Theorem~\ref{thm:pdo} for $1\leq p\leq n-2$.
Throughout this section, the metric is fixed, and we suppress the
subscript $g$ on $\Lambda$.

\subsection{Boundary operators and the coclosedness constraint}

Decompose the full boundary map in \eqref{eq:full-DN} into blocks:
\[
  \mathcal N=
  \begin{pmatrix}
    \mathcal N_{11}&\mathcal N_{12}\\
    \mathcal N_{21}&\mathcal N_{22}
  \end{pmatrix}
  \quad\text{on }\Omega^p(\Sigma)\oplus\Omega^{p-1}(\Sigma).
\]
The scalar principal symbol in \eqref{eq:full-DN-symbol} implies that the
diagonal blocks have order one and principal symbol $|\xi|\Id$,
whereas the off-diagonal blocks have order at most zero.

For the collar decomposition \eqref{eq:collar-form}, the trace
identities are
\begin{equation}\label{eq:trace-identities}
\begin{split}
  i^*\delta u&=\delta_\Sigma a(0)+b'(0)+Zb(0),\\
  i^*\iota_\nu du&=-a'(0)-d_\Sigma b(0),
\end{split}
\end{equation}
where $Z$ is a smooth bundle endomorphism on $(p-1)$-forms determined
by the first normal derivative of $h_s$. The second identity follows by
differentiating $a-ds\wedge b$. For the first, use
$\delta=-\sum_j\iota_{e_j}\nabla_{e_j}$ in an adapted orthonormal frame.
The tangential derivatives of $a$ give $\delta_\Sigma a$, the normal
derivative gives $b'$, and the remaining terms arise from the second
fundamental form and are zeroth order in $b$.

Define
\begin{equation}\label{eq:boundary-blocks}
\begin{aligned}
  A\beta&=-i^*\delta H(0,\beta),&
  B\varphi&=i^*\delta H(\varphi,0),\\
  R\varphi&=i^*\iota_\nu dH(\varphi,0),&
  S\beta&=i^*\iota_\nu dH(0,\beta).
\end{aligned}
\end{equation}
Thus $A$ acts on $(p-1)$-forms, $R$ acts on $p$-forms, and $B,S$
change degree by $-1,+1$, respectively. By linearity,
\begin{equation}\label{eq:constraint}
  i^*\delta H(\varphi,\beta)=B\varphi-A\beta,\qquad
  i_ndH(\varphi,\beta)=R\varphi+S\beta.
\end{equation}
By Lemma~\ref{lem:coclosed}, the coclosedness constraint is therefore
the boundary equation $A\beta=B\varphi$.

\begin{lemma}\label{lem:block-symbols}
The operators in \eqref{eq:boundary-blocks} are classical
pseudodifferential operators of order at most one. Their principal symbols
are
\begin{equation}\label{eq:block-symbols}
\begin{aligned}
  \sigma_1(A)&=|\xi|\Id,&
  \sigma_1(B)&=-\ii\iota_{\xi^\sharp},\\
  \sigma_1(R)&=|\xi|\Id,&
  \sigma_1(S)&=-\ii\xi\wedge.
\end{aligned}
\end{equation}
In particular, $A$ is elliptic.
\end{lemma}
\begin{proof}
Equations~\eqref{eq:trace-identities} and~\eqref{eq:boundary-blocks} give
\[
  A=\mathcal N_{22}-Z,\quad
  B=\delta_\Sigma-\mathcal N_{21},\quad
  R=\mathcal N_{11},\quad
  S=\mathcal N_{12}-d_\Sigma.
\]
Use \eqref{eq:full-DN-symbol}, the order-zero property of the
off-diagonal blocks, and the symbol conventions
$\sigma_1(d_\Sigma)=\ii\xi\wedge$ and
$\sigma_1(\delta_\Sigma)=-\ii\iota_{\xi^\sharp}$.
\end{proof}

\begin{lemma}\label{lem:A-kernel}
The operator $A$ is nonnegative and formally self-adjoint. Its kernel
$K$ is a finite-dimensional space of smooth forms. More precisely,
\begin{equation}\label{eq:A-kernel}
  K=\{i^*\iota_\nu h:h\in\HD^p(M)\},
\end{equation}
and $H(0,\kappa)\in\HD^p(M)$ for every $\kappa\in K$.
\end{lemma}
\begin{proof}
For $\beta,\gamma\in\Omega^{p-1}(\Sigma)$, apply
\eqref{eq:Green-energy} to $u=H(0,\beta)$ and $v=H(0,\gamma)$.
Since $\Delta u=0$ and $i^*v=0$, we obtain
\begin{equation}\label{eq:A-energy}
  \ip{A\beta}{\gamma}{\Sigma}
  =\ip{du}{dv}{M}+\ip{\delta u}{\delta v}{M}.
\end{equation}
This proves symmetry and nonnegativity. Ellipticity on the closed boundary
gives a finite-dimensional smooth kernel. If $\kappa\in K$, then
\eqref{eq:A-energy} with $\beta=\gamma=\kappa$ gives
\[
  dH(0,\kappa)=0,\qquad \delta H(0,\kappa)=0.
\]
The tangential trace is zero, so $H(0,\kappa)\in\HD^p(M)$.
Conversely, if $h\in\HD^p(M)$ and $\kappa=\iota_\nu h$, uniqueness of the full
Dirichlet problem gives $H(0,\kappa)=h$, and hence $A\kappa=0$.
\end{proof}

Let $P_K$ be the $L^2$-orthogonal projection onto $K$. It is smoothing
because $K$ has a finite smooth basis. The operator $A+P_K$ is elliptic,
self-adjoint, and invertible. Indeed, its restriction to $K$ is the
identity, while $A$ is positive and invertible on $K^\perp$.
Equivalently, injectivity follows by applying $P_K$ to $(A+P_K)\beta=0$,
and surjectivity follows from the self-adjoint elliptic Fredholm
alternative. Thus
\begin{equation}\label{eq:A-inverse}
  G_A=(A+P_K)^{-1}\in\Psicl{-1},\qquad
  \sigma_{-1}(G_A)=|\xi|^{-1}\Id.
\end{equation}
The inverse is a bounded map $H^s\to H^{s+1}$ for every $s$.
A classical parametrix differs from this inverse by a smoothing operator,
by elliptic regularity. The exact identities are
\begin{equation}\label{eq:A-inverse-identities}
  AG_A=G_AA=\Id-P_K,\qquad G_AP_K=P_KG_A=P_K.
\end{equation}

\subsection{Compatibility and the exact Schur complement}

\begin{proposition}\label{prop:Schur}
The compatibility and adjoint identities
\begin{equation}\label{eq:compatibility}
  P_KB=0,\qquad S=-B^*
\end{equation}
hold on smooth boundary forms. Moreover,
\begin{equation}\label{eq:Schur}
  \Lambda=R+SG_AB=R-B^*G_AB.
\end{equation}
\end{proposition}
\begin{proof}
Let $\kappa\in K$, set $h=H(0,\kappa)$, and take $v=H(\varphi,0)$.
By Lemma~\ref{lem:A-kernel}, $dh=\delta h=0$.
Formula~\eqref{eq:Green-energy}, applied to $v,h$, gives
\[
  0=\ip{dv}{dh}{M}+\ip{\delta v}{\delta h}{M}
   =-\ip{B\varphi}{\kappa}{\Sigma}.
\]
Thus $B\varphi\perp K$.

Next let $u=H(0,\beta)$ and $v=H(\varphi,0)$. Applying
\eqref{eq:Green-energy} in the two orders gives
\[
  \ip{S\beta}{\varphi}{\Sigma}
  =\ip{du}{dv}{M}+\ip{\delta u}{\delta v}{M}
  =-\ip{\beta}{B\varphi}{\Sigma},
\]
which proves $S=-B^*$.

For any $\varphi$, put $\beta=G_AB\varphi$.
Equations~\eqref{eq:A-inverse-identities} and~\eqref{eq:compatibility} imply
\[
  A\beta=B\varphi,\qquad P_K\beta=0.
\]
Thus $\beta$ is the unique solution of the constraint in $K^\perp$.
The harmonic form $\widehat\varphi=H(\varphi,\beta)$ satisfies
\[
  i^*\widehat\varphi=\varphi,\qquad i^*\delta\widehat\varphi=0.
\]
By Lemma~\ref{lem:coclosed}, $\delta\widehat\varphi=0$.
It is therefore an admissible extension in the definition of $\Lambda$.
The second identity in \eqref{eq:constraint} gives \eqref{eq:Schur}.
\end{proof}

The finite-dimensional correction in \eqref{eq:A-inverse} is essential
to the global formula. A principal-symbol inversion alone would not
justify solving $A\beta=B\varphi$ in the presence of relative harmonic
fields. The compatibility identity removes precisely this obstruction.

\subsection{The principal symbol and the coclosed realization}

\begin{proof}[Proof of Theorem~\ref{thm:pdo}]
For $1\leq p\leq n-2$, Proposition~\ref{prop:Schur} and the composition
calculus give $\Lambda\in\Psicl{1}$.
By \eqref{eq:block-symbols} and \eqref{eq:A-inverse},
\begin{equation}\label{eq:Lambda-symbol}
\begin{split}
  \sigma_1(\Lambda)(x,\xi)
  &=|\xi|\Id+(-\ii\xi\wedge)|\xi|^{-1}(-\ii\iota_{\xi^\sharp})\\
  &=|\xi|\Id-|\xi|^{-1}\xi\wedge\iota_{\xi^\sharp}
   =|\xi|\pi_p(x,\xi).
\end{split}
\end{equation}
To identify the projection, use
\[
  G_p=(\Delta_\Sigma+H_p)^{-1}-H_p.
\]
The elliptic inverse construction gives $G_p\in\Psicl{-2}$, with
principal symbol $|\xi|^{-2}\Id$. Thus \eqref{eq:projection} implies
\[
  \Pi_p\in\Psicl{0},\qquad
  \sigma_0(\Pi_p)=\Id-|\xi|^{-2}\xi\wedge\iota_{\xi^\sharp}=\pi_p.
\]
The identity
\[
  \iota_{\xi^\sharp}(\xi\wedge\alpha)
  +\xi\wedge\iota_{\xi^\sharp}\alpha=|\xi|^2\alpha
\]
shows that $\pi_p$ is the orthogonal projection onto
$\ker\iota_{\xi^\sharp}$. Its rank is $\binom{n-2}{p}$.
The exact compression identity follows from Lemma~\ref{lem:compression}.

Let $\mathcal D_p=(\Id+\Delta_{\Sigma,p})^{1/2}$.
This operator is positive and classical of order one, has principal symbol
$|\xi|\Id$, and commutes with $\Pi_p$; see \cite{Seeley1967}.
Define
\begin{equation}\label{eq:elliptic-completion}
  \mathcal L=\Lambda+(\Id-\Pi_p)\mathcal D_p(\Id-\Pi_p).
\end{equation}
Its principal symbol is
\[
  |\xi|\pi_p+|\xi|(\Id-\pi_p)=|\xi|\Id.
\]
Thus $\mathcal L$ is a symmetric nonnegative elliptic operator of order one
on the full bundle. Choose a two-sided parametrix
$\mathcal Q_0\in\Psicl{-1}$ and set
$\mathcal Q_p=\Pi_p\mathcal Q_0\Pi_p$.
Since $\mathcal L\Pi_p=\Pi_p\mathcal L=\Lambda$, compression of the
parametrix identities gives \eqref{eq:parametrix-main}, with smoothing
remainders.

The self-adjoint realization of $\mathcal L$ has domain
$H^1\Omega^p(\Sigma)$ and compact resolvent, by elliptic regularity and
the compact embedding $H^1\hookrightarrow L^2$.
The bounded projection $\Pi_p$ reduces this realization, since it commutes
with $\mathcal L$ on smooth forms and preserves $H^1$.
Its restriction to $\im\Pi_p$ is the self-adjoint realization of $\Lambda$
with domain $H^1\Omega^p(\Sigma)\cap\ker\delta_\Sigma$.
Its resolvent is compact, being the restriction of a compact resolvent
to a reducing subspace.

For $p=0$, the extension problem is the scalar harmonic Dirichlet problem,
$\Pi_0=\Id$, and $\Lambda$ is the classical elliptic Dirichlet-to-Neumann
operator. The same symbol and realization statements follow directly
from the scalar case of \eqref{eq:full-DN-symbol}.
\end{proof}

\subsection{Proof of Corollary~\ref{cor:Weyl}}

We now provide the proof of the Weyl remainder estimate stated in Corollary~\ref{cor:Weyl}.

\begin{proof}[Proof of Corollary~\ref{cor:Weyl}]
We first recall the corresponding boundary Hodge count.
Let $F_q(\tau)$ count the positive eigenvalues of
$\sqrt{\Delta_{\Sigma,q}}$ below $\tau$, with multiplicity, and let
$C_q(\tau)$ count those on coexact $q$-forms.
The standard Weyl estimate for the Hodge Laplacian on the closed smooth
boundary gives
\[
  F_q(\tau)=\binom {n-1}q c_\Sigma\tau^{n-1}+O(\tau^{n-2}),\qquad
  c_\Sigma=\frac{\omega_{n-1}\Vol_{g_\Sigma} (\Sigma)}{(2\pi)^{n-1}};
\]
see \cite{Hormander1968,Hormander2007}.
For every positive Hodge eigenvalue $\mu$, $d_\Sigma$ identifies its
coexact $(q-1)$-form eigenspace with its exact $q$-form eigenspace.
Its inverse on that eigenspace is $\mu^{-1}\delta_\Sigma$.
The Hodge decomposition therefore gives
\[
  F_q=C_q+C_{q-1},\qquad C_{-1}=0.
\]
Induction and the alternating-binomial identity imply
\begin{equation}\label{eq:coexact-count}
  C_p(\tau)=\sum_{q=0}^p(-1)^{p-q}F_q(\tau)
           =\binom{n-2}{p}c_\Sigma\tau^{n-1}+O(\tau^{n-2}).
\end{equation}
Adding the finite-dimensional harmonic space does not change this estimate.

The operators $\Lambda$ and
$\Pi_p(\Id+\Delta_{\Sigma,p})^{1/2}\Pi_p$ have the same order-one
principal symbol. Their difference is therefore a bounded self-adjoint
operator of order zero. On the coclosed space, the second operator has
eigenvalues $\sqrt{1+\lambda_j^{\prime(p)}}$.
The variational principle yields a constant $C$ such that
\[
  \left|\sigma_j^{(p)}-\sqrt{1+\lambda_j^{\prime(p)}}\right|
  \leq C\qquad(j\geq1).
\]
The count for $(\Id+\Delta_{\Sigma,p})^{1/2}$ has the same leading term and
remainder as \eqref{eq:coexact-count}. Comparing its counts at $\tau-C$
and $\tau+C$ proves \eqref{eq:Weyl}.
\end{proof}

Girouard, Karpukhin, Levitin, and Polterovich
\cite[Theorems~5.6 and~5.8]{GKLP2022} established the bounded eigenvalue
comparison and the leading Weyl term by a different method.
The proof above uses the classical symbol to recover the comparison and
the $O(\tau^{n-2})$ remainder; the latter was explicitly requested in
\cite[Open Question~7.23]{CGGS2024}.

\section{Conformal non-rigidity}\label{sec:conformal}

The relevant invariance holds for every critical-dimensional filling.

\begin{proposition}\label{prop:conformal}
Let $n=2p+2$ and $\widetilde g=e^{2u}g$ on $M$, where
$u\in C^\infty(M)$ and $u|_\Sigma=0$.
Then, under the identity identification of boundary forms,
\begin{equation}\label{eq:operator-invariance}
  \Lambda_{\widetilde g}=\Lambda_g.
\end{equation}
\end{proposition}
\begin{proof}
For every $q$-form $\eta$,
\[
  |\eta|_{\widetilde g}^2\,dV_{\widetilde g}
  =e^{(n-2q)u}|\eta|_g^2\,dV_g.
\]
Apply this to $\eta=d\omega$, of degree $p+1=n/2$.
The interior energy in \eqref{eq:energy-minimum} is unchanged for every
trial extension. The trace constraint is metric-independent, while
$u|_\Sigma=0$ fixes the boundary metric and its $L^2$ inner product.
It follows that $q_{\widetilde g}(\varphi)=q_g(\varphi)$ for all smooth
boundary forms. Polarization and \eqref{eq:boundary-energy} yield
\eqref{eq:operator-invariance}.
\end{proof}

Karpukhin's Proposition~2.10 gives the corresponding spectral invariance.
The preceding argument records that the complete boundary operator is
preserved.

\begin{proof}[Proof of Theorem~\ref{thm:nonrigidity}]
Apply Proposition~\ref{prop:conformal} to the Euclidean ball $(\mathbb{B}^n,g_E)$.
The induced boundary metric, and hence the entire boundary Hodge spectrum,
is unchanged. For $p\geq1$, the ball computation in
\cite[Theorem~8.1]{Karpukhin2019} gives
\begin{equation}\label{eq:ball-spectrum}
\begin{aligned}
  \sigma_1^{(p)}(\mathbb{B}^{2p+2})=\cdots=\sigma_{D_p}^{(p)}(\mathbb{B}^{2p+2})&=p+1,\\
  \lambda_1^{\prime(p)}(\mathbb{S}^{2p+1})=\cdots=
  \lambda_{D_p}^{\prime(p)}(\mathbb{S}^{2p+1})&=(p+1)^2.
\end{aligned}
\end{equation}
Since $I_p=b_{p+1}=0$ on the ball, these identities give equality in
\eqref{eq:HPS} for every $1\leq m\leq D_p/2$.
For $p=0$, $D_0=2$, $I_0=1$, and $b_1=0$; the assertion is the disk
identity $(\sigma_2^{(0)})^2=\lambda_2^{\prime(0)}=1$.
Thus every conformal change fixing the boundary metric preserves all the
required equalities.

Now take a nonzero $u\in C_c^\infty( \mathbb{B}^n)$ with $u\geq0$.
The metric is Euclidean on an entire boundary collar.
If it were isometric to a Euclidean ball of radius $R$, the boundary
area would give
\[
  |\mathbb{S}^{n-1}|_{\Tilde{g}|_{\mathbb{S}^{n-1}}}=R^{n-1}|\mathbb{S}^{n-1}|_{g_{\mathbb{S}^{n-1}}},
\]
so $R=1$. But
\[
  \Vol_{e^{2u}g_E}(\mathbb{B}^n)=\int_{\mathbb{B}^n}e^{nu}\,dx
                       >\Vol_{g_E}(\mathbb{B}^n),
\]
a contradiction. Finally, for fixed such $u$, the volume
$\int_{\mathbb{B}^n}e^{ntu}\,dx$ is strictly increasing in $t>0$.
The metrics in the stated one-parameter family are therefore pairwise
nonisometric.
\end{proof}

\begin{remark}
These examples have the same boundary metric and all the same boundary
jets as the Euclidean metric. A rigidity statement must therefore
account for conformal changes fixing the boundary metric or impose
additional restrictions on the interior geometry.
No classification modulo this conformal freedom is proved here.
\end{remark}

\section{A shrinking core with Euclidean annuli}\label{sec:construction}

For the remainder of the geometric argument, fix $p\geq1$ and $N\geq1$, We denote the unit round metric on $\mathbb{S}^{n-1}$ by $g_{\mathbb{S}}$.

\subsection{The core and the interface metric}

Choose $N$ pairwise disjoint small closed geodesic balls
$\mathbb{B}_1^n,\ldots,\mathbb{B}_N^n$ in an auxiliary round $\mathbb{S}^n$, and let
\[
  C=\mathbb{S}^n\setminus\bigcup_{j=1}^N\Int \mathbb{B}_j^n,\qquad
  \Gamma_j=\partial \mathbb{B}_j^n.
\]
The core $C$ is compact, connected, and oriented.
Its boundary has $N$ components.
Choose diffeomorphisms $\kappa_j:\mathbb{S}^{n-1}\to\Gamma_j$.
By the collar theorem \cite[Theorem~9.25]{Lee2012}, there are mutually
disjoint inward collars
\[
  \Theta_j:[0,2s_0)\times \mathbb{S}^{n-1}\longrightarrow U_j\subset C,\qquad
  \Theta_j(0,\theta)=\kappa_j(\theta),\qquad 0<s_0<\tfrac14.
\]
Choose a smooth metric $g_C$ such that, on each smaller collar,
\begin{equation}\label{eq:core-collar}
  \Theta_j^*g_C=ds^2+(1-s)^2g_{\mathbb{S}},\qquad 0\leq s\leq s_0.
\end{equation}
To construct it, start with any smooth background metric.
In each disjoint collar interpolate by a smooth cutoff between the
background tensor and
$(\Theta_j^{-1})^*(ds^2+(1-s)^2g_{\mathbb{S}})$, taking the cutoff to be one on
$[0,s_0]$ and zero near $2s_0$.
A convex combination of positive-definite tensors is positive definite.
The cutoff is zero near the outer edge of the chart, so the resulting
metric is globally smooth.
This prescribes $\kappa_j^*(g_C|_{T\Gamma_j})=g_{\mathbb{S}}$; the auxiliary round metric
used to locate the balls need not induce this boundary metric.

For $0<\epsilon<1$, equip the core with $\epsilon^2g_C$ and attach
\[
  A_\epsilon^{(j)}=[\epsilon,1]\times \mathbb{S}^{n-1},\qquad
  g_A=dr^2+r^2g_{\mathbb{S}},
\]
by identifying $(\epsilon,\theta)$ with $\kappa_j(\theta)$.
Call the resulting space and metric $(M_\epsilon,g_\epsilon)$.

\begin{lemma}\label{lem:geometry}
For every $0<\epsilon<1$, $(M_\epsilon,g_\epsilon)$ is a smooth compact
connected oriented Riemannian manifold.
Its boundary is isometric to the disjoint union of $N$ unit round
$(n-1)$-spheres.
\end{lemma}
\begin{proof}
At an interface, put
\[
  x=r-\epsilon\quad\text{on the annulus},\qquad
  x=-\epsilon s\quad\text{on the core collar}.
\]
Together with local spherical coordinates, this defines the smooth
gluing charts. On both sides, \eqref{eq:core-collar} gives the same
metric expression
\begin{equation}\label{eq:seam-metric}
  g_\epsilon=dx^2+(\epsilon+x)^2g_{\mathbb{S}}.
\end{equation}
For each fixed $\epsilon$, choose a two-sided neighborhood sufficiently
small that $\epsilon+x>0$ and both sides remain inside the stated
coordinate neighborhoods. Formula~\eqref{eq:seam-metric} is a single
smooth positive-definite tensor on that neighborhood. Thus the metrics
match to every order.

Orient each annulus so that the inner-boundary gluing reverses the
induced boundary orientations. These orientations extend that of $C$.
Every annulus meets the connected core, so $M_\epsilon$ is connected.
Compactness is immediate. Only the outer spheres $r=1$ remain as
boundary, and each carries $h$.
\end{proof}

The diameter of the core tends to zero as $\epsilon\downarrow0$.
No curvature bound uniform in $\epsilon$ is required.
The construction uses the scaled-core and conical-collar geometry of
\cite{AnneTakahashi2015,AnneTakahashi2024}; the outer spherical boundaries
are retained here for the Steklov problem.

\subsection{Cohomology and the first boundary eigenspace}

\begin{lemma}\label{lem:indices-spectrum}
For the preceding family,
\begin{equation}\label{eq:cohomology}
  H^q(M_\epsilon;\R)\cong
  \begin{cases}
    \R,&q=0,\\
    \R^{N-1},&q=n-1,\\
    0,&\text{otherwise}.
  \end{cases}
\end{equation}
In particular, $I_p(M_\epsilon)=b_{p+1}(M_\epsilon)=0$.
Moreover,
\begin{equation}\label{eq:boundary-spectrum}
  \lambda_j^{\prime(p)}(\partial M_\epsilon)=(p+1)^2,\qquad
  1\leq j\leq ND_p.
\end{equation}
\end{lemma}
\begin{proof}
Attaching annuli only extends the boundary collars, so $M_\epsilon$ is
diffeomorphic to $C$. Removing one center point from each excised ball
and radially retracting that punctured ball onto its boundary gives a
deformation retraction of $\mathbb{S}^n\setminus\{x_1,\ldots,x_N\}$ onto $C$.
Alexander duality and the universal coefficient theorem over $\R$
give \eqref{eq:cohomology}; see
\cite[Corollary~3.45]{Hatcher2002}.
Since $p\geq1$,
\[
  1\leq p<p+1\leq2p=n-2.
\]
Thus $H^p(M_\epsilon;\R)=H^{p+1}(M_\epsilon;\R)=0$, and the definitions
\eqref{eq:topological-indices} imply $I_p=b_{p+1}=0$.

The positive coclosed $p$-form Hodge eigenvalues on the unit sphere
$\mathbb{S}^{2p+1}$ are $(k+p)^2$, $k\geq1$; see
\cite{IkedaTaniguchi1978} and
\cite[Theorem~8.1, second family]{Karpukhin2019}.
The first eigenspace has dimension $D_p$, and there are no harmonic
$p$-forms in this degree. The spectrum on the disjoint union of $N$
spheres is the direct sum of the component spectra, proving
\eqref{eq:boundary-spectrum}.
\end{proof}

\section{Annular energy estimates and sharpness}\label{sec:energy}

We prove Theorem~\ref{thm:cluster} by two independent energy bounds.
Both depend only on the Euclidean annuli and their outer traces.

\subsection{Conformal reduction and coclosed projection}

On one annulus $A_\epsilon=A_\epsilon^{(j)}$, set $t=-\log r$ and
$T=\log(1/\epsilon)$. Then
\[
  g_A=e^{-2t}\bar{g},\qquad
  \bar{g}=dt^2+g_{\mathbb{S}},\qquad 0\leq t\leq T.
\]
The conformal scaling calculation in Proposition~\ref{prop:conformal}
gives
\begin{equation}\label{eq:annular-conformal}
  \int_{A_\epsilon}|d\omega|_{g_A}^2\,dV_{g_A}
  =\int_{[0,T]\times \mathbb{S}^{2p+1}}|d\omega|_{\bar{g}}^2
                             \,dV_{\bar{g}}.
\end{equation}
Write a smooth $p$-form as $\omega=\alpha(t)+dt\wedge\beta(t)$, where
$\alpha(t)$ and $\beta(t)$ are spherical forms of degrees $p$ and
$p-1$, respectively. Let $d_\mathbb{S}$ denote the spherical exterior derivative.
Then
\[
  d\omega=d_\mathbb{S}\alpha+dt\wedge(\alpha'-d_\mathbb{S}\beta),\qquad
  \alpha'=\partial_t\alpha.
\]
The two terms are orthogonal for the product metric, so
\begin{equation}\label{eq:annular-energy}
  \int_{A_\epsilon}|d\omega|_{g_A}^2\,dV_{g_A}
  =\int_0^T\bigl(\|d_\mathbb{S}\alpha(t)\|_{L^2(\mathbb{S}^{n-1},g_{\mathbb{S}})}^2
                +\|\alpha'(t)-d_\mathbb{S}\beta(t)\|_{L^2(\mathbb{S}^{n-1},g_{\mathbb{S}})}^2\bigr)\,dt.
\end{equation}

Let $P$ be the coclosed Hodge projection on $(\mathbb{S}^{2p+1},h)$ and put
$u(t)=P\alpha(t)$. Since $P$ is independent of $t$,
\eqref{eq:projection-properties} gives
\[
  d_\mathbb{S}\alpha=d_\mathbb{S}u,\qquad P(\alpha'-d_\mathbb{S}\beta)=u'.
\]
Using the first coclosed spherical eigenvalue $(p+1)^2$ and the contraction
property of $P$ yields
\begin{equation}\label{eq:projected-energy}
  \int_{A_\epsilon}|d\omega|_{g_A}^2\,dV_{g_A}
  \geq\int_0^T\bigl(\|u'\|_{L^2(\mathbb{S}^{n-1},g_{\mathbb{S}})}^2+(p+1)^2\|u\|_{L^2(\mathbb{S}^{n-1},g_{\mathbb{S}})}^2\bigr)\,dt.
\end{equation}
If the outer trace of $\omega$ is coclosed and equals $\varphi$, then
$u(0)=\varphi$. 

\subsection{The free-end lower bound}

\begin{lemma}\label{lem:free-end}
Let $\mathsf H$ be a Hilbert space, $T>0$, and
$u\in H^1((0,T);\mathsf H)$. Then
\begin{equation}\label{eq:square-completion}
  \int_0^T\bigl(\|u'\|^2+(p+1)^2\|u\|^2\bigr)\,dt
  =(p+1)\tanh((p+1)T)\|u(0)\|^2+\int_0^T\|u'+bu\|^2\,dt,
\end{equation}
which directly yields
\begin{equation}\label{eq:free-end}
  \int_0^T\bigl(\|u'\|^2+(p+1)^2\|u\|^2\bigr)\,dt
  \geq (p+1)\tanh((p+1)T)\|u(0)\|^2.
\end{equation}
Equality is attained by
$u(t)=\cosh((p+1)(T-t))\varphi/\cosh((p+1)T)$ for prescribed $u(0)=\varphi$.
\end{lemma}
\begin{proof}
Set $b(t)=(p+1)\tanh((p+1)(T-t))$. Then $b'=b^2-(p+1)^2$, $b(T)=0$, and
$b(0)=(p+1)\tanh((p+1)T)$. Integration by parts gives \eqref{eq:square-completion}.

The displayed equality function solves
$u'+bu=0$ and has the required initial value. Its derivative vanishes
at $T$, which is the natural free-end condition.
\end{proof}

\begin{proposition}\label{prop:lower}
For every smooth coclosed boundary form $\varphi$ on $\partial M_\epsilon$,
\begin{equation}\label{eq:global-lower}
  q_{g_\epsilon}(\varphi)
  \geq (p+1)\tanh((p+1)T)\|\varphi\|_{L^2(\partial M_\epsilon)}^2.
\end{equation}
Consequently, $\sigma_1^{(p)}(M_\epsilon,g_\epsilon)\geq (p+1)\tanh((p+1)T)$.
\end{proposition}
\begin{proof}
Let $\eta$ be any smooth extension of $\varphi$.
Apply \eqref{eq:projected-energy} and Lemma~\ref{lem:free-end} on each
annulus. Summing and discarding only the nonnegative core energy gives
\[
  \int_{M_\epsilon}|d\eta|_{g_\epsilon}^2\,dV
  \geq (p+1)\tanh((p+1)T)\sum_{j=1}^N\|\varphi_j\|_{L^2(\mathbb{S}^n,g_{\mathbb{S}})}^2.
\]
Take the minimum over all extensions as in \eqref{eq:energy-minimum},
and then use \eqref{eq:variational} with $k=1$.
\end{proof}

\subsection{The zero-end upper bound}

The lower bound together with \eqref{eq:HPS} suffices for sharpness.
The independent upper estimate below determines the limit of the entire
first $ND_p$ cluster. Let $E_j$ be the first coclosed $p$-form Hodge
eigenspace on the $j$th boundary sphere, and put
\[
  E=\bigoplus_{j=1}^NE_j,\qquad \dim E=ND_p.
\]
Every $\varphi=(\varphi_1,\ldots,\varphi_N)\in E$ satisfies
$\|d_\mathbb{S}\varphi_j\|_h^2=(p+1)^2\|\varphi_j\|_h^2$.

\begin{proposition}\label{prop:upper}
For every $0<\epsilon<1$,
\begin{equation}\label{eq:global-upper}
  \sigma_{ND_p}^{(p)}(M_\epsilon,g_\epsilon)\leq (p+1)\coth((p+1)T).
\end{equation}
\end{proposition}
\begin{proof}
The scalar minimizer with endpoint values $f(0)=1$ and $f(T)=0$ is
\begin{equation}\label{eq:sinh-profile}
  f_T(t)=\frac{\sinh((p+1)(T-t))}{\sinh((p+1)T)}.
\end{equation}
It satisfies $f_T''=(p+1)^2f_T$, and hence
\begin{equation}\label{eq:profile-energy}
  \int_0^T\bigl((f_T')^2+(p+1)^2f_T^2\bigr)\,dt
  =[f_Tf_T']_0^T=-f_T'(0)=(p+1)\coth((p+1)T).
\end{equation}
The form $f_T(t)\varphi_j$ vanishes at the inner seam, but its normal
derivatives need not vanish there. To obtain globally smooth trial forms,
choose $\chi\in C^\infty(\R;[0,1])$ with
$\chi=0$ on $(-\infty,1/3]$ and $\chi=1$ on $[2/3,\infty)$, and set
\[
  f_{T,\tau}(t)=\chi\bigl((T-t)/\tau\bigr)f_T(t),\qquad 0<\tau<T/2.
\]
For fixed $T$, near $T$ one has $|f_T(t)|\leq C_T(T-t)$ and
$|f_T'(t)|\leq C_T$ for some constant $C_T$. The support of $f_{T,\tau}-f_T$ has length at most
$2\tau/3$. Differentiating the cutoff therefore gives
\[
  \|f_{T,\tau}-f_T\|_{L^2(0,T)}^2=O_T(\tau^3),\qquad
  \|f_{T,\tau}'-f_T'\|_{L^2(0,T)}^2=O_T(\tau).
\]
Thus $f_{T,\tau}\to f_T$ in $H^1(0,T)$.

For $\varphi\in E$, define $\eta_\tau=f_{T,\tau}(t)\varphi_j$ on the
$j$th annulus and $\eta_\tau=0$ on the core.
This is smooth across the interfaces, since it vanishes in their
neighborhoods, and its outer trace is $\varphi$. By \eqref{eq:annular-energy},
\begin{equation}\label{eq:trial-energy}
\begin{split}
  \int_{M_\epsilon}|d\eta_\tau|_{g_{\epsilon}}^2\,dV
    &=J_{T,\tau}\|\varphi\|_{L^2(\partial M_\epsilon, g_{\epsilon}|_{\partial M_\epsilon})}^2,\\
  J_{T,\tau}&=\int_0^T\bigl((f_{T,\tau}')^2+(p+1)^2f_{T,\tau}^2\bigr)\,dt.
\end{split}
\end{equation}
The factorization is exact for every $\varphi\in E$, so the estimate is
uniform on its unit sphere. Given any $(ND_p-1)$-dimensional subspace
$F\subset\cC^p(\partial M_\epsilon)$, choose
$0\ne\varphi\in E\cap F^\perp$. The energy minimum
\eqref{eq:energy-minimum} gives
$q_{g_\epsilon}(\varphi)/\|\varphi\|_{L^2(\partial M_\epsilon, g_\epsilon|_{\partial M_\epsilon})}^2\leq J_{T,\tau}$.
The max--min principle \eqref{eq:variational} therefore implies
\[
  \sigma_{ND_p}^{(p)}(M_\epsilon,g_\epsilon)\leq J_{T,\tau}.
\]
Let $\tau\downarrow0$ with $\epsilon$, and therefore $T$, fixed.
Equation~\eqref{eq:profile-energy} gives \eqref{eq:global-upper}.
\end{proof}

\begin{proof}[Proof of Theorem~\ref{thm:cluster}]
The geometric and topological assertions follow from
Lemmas~\ref{lem:geometry} and~\ref{lem:indices-spectrum}.
By Propositions~\ref{prop:lower} and~\ref{prop:upper} and eigenvalue ordering,
\[
  (p+1)\tanh((p+1)T)\leq\sigma_1^{(p)}
  \leq\sigma_j^{(p)}\leq\sigma_{ND_p}^{(p)}\leq (p+1)\coth((p+1)T)
  \qquad(1\leq j\leq ND_p).
\]
As $\epsilon\downarrow0$, $T\to\infty$ and both endpoints tend to $p+1$.
This proves \eqref{eq:cluster-bounds} and the cluster convergence.
The argument has not used \eqref{eq:HPS}.
\end{proof}

\subsection{Sharpness and the spectral indices}

\begin{proof}[Proof of Corollary~\ref{cor:sharpness} for $p\geq1$]
Choose $N$ as in \eqref{eq:Dp}. Then
\[
  1\leq m\leq2m-1\leq ND_p.
\]
By Lemma~\ref{lem:indices-spectrum}, $I_p=b_{p+1}=0$ and
$\lambda_{2m-1}^{\prime(p)}(\partial M_\epsilon)=(p+1)^2$.
By Theorem~\ref{thm:cluster},
$\sigma_m^{(p)}(M_\epsilon,g_\epsilon)\to p+1$.
This proves asymptotic equality with a fixed positive right-hand side.

For the stated error bound, apply Karpukhin's inequality \eqref{eq:HPS}
to obtain $\sigma_m^{(p)}\leq p+1$.
Combining this with the independent lower bound gives
\[
\begin{split}
  0\leq (p+1)^2-\bigl(\sigma_m^{(p)}\bigr)^2
  &\leq (p+1)^2\bigl(1-\tanh^2((p+1)T)\bigr)\\
  &=(p+1)^2\operatorname{sech}^2((p+1)T)
    =\frac{4(p+1)^2\epsilon^{2(p+1)}}{(1+\epsilon^{2(p+1)})^2}.
\end{split}
\]
\end{proof}

\subsection{The established planar case}

For completeness, we explain how the known scalar sharpness theorem fits
the indexing and smoothness conventions used here. This part is due to
Girouard and Polterovich
\cite[Theorem~1.3.1]{GirouardPolterovich2010}.

Write $s_m(\Omega)$ for the $m$th positive scalar Steklov eigenvalue of
a connected planar domain, and $L(\partial\Omega)$ for its perimeter.
For each $m\geq1$, Girouard and Polterovich construct simply connected
bounded Lipschitz domains $\Omega_\ell$ such that
\begin{equation}\label{eq:planar-product}
  s_m(\Omega_\ell)L(\partial\Omega_\ell)\longrightarrow2\pi m.
\end{equation}
For $m=1$, the disk gives equality directly. For $m>1$, the construction
degenerates toward $m$ equal disks.

For each fixed $\ell$, smooth the finitely many boundary corners from
the interior in equi-Lipschitz charts.
The resulting simply connected smooth domains
$\Omega_{\ell,q}\subset\Omega_\ell$ may be chosen so that their indicators
converge in $L^1$, their perimeters converge, and their
$BV$-to-boundary-$L^1$ trace norms are uniformly bounded in $q$.
These fixed-domain approximations and Steklov convergence are justified by
\cite[\S3.3.2 and Proposition~3.5]{KarpukhinLagace2024}.
The constants need not be uniform in $\ell$.
A diagonal choice preserves \eqref{eq:planar-product} for smooth domains.

Rescale each chosen domain so that its boundary length is $2\pi m$.
The product in \eqref{eq:planar-product} is scale invariant, so $s_m\to1$.
A smooth simply connected planar domain has $I_0=1$ and $b_1=0$, while
its boundary is intrinsically a circle of length $2\pi m$.
With zeros included in both spectral sequences,
\[
  \sigma_{m+1}^{(0)}=s_m\longrightarrow1,\qquad
  \lambda_{I_0+b_1+2m-1}^{\prime(0)}
  =\lambda_{2m}^{\prime(0)}
  =\left(\frac{2\pi m}{L(\partial\Omega)}\right)^2=1.
\]
This proves the $p=0$ assertion in Corollary~\ref{cor:sharpness} and
completes the all-degree sharpness conclusion.

\section*{ACKNOWLEDGEMENTS}
The authors acknowledge the use of AI to assist with English-language editing,
manuscript polishing, and preliminary exploration of mathematical proofs. The authors
take full responsibility for all mathematical content, arguments, and the final version of
the manuscript.

\end{document}